\documentclass[12pt]{amsart}

\usepackage{hyperref,graphicx}
\usepackage{xcolor}
\usepackage{url}
\usepackage{wrapfig}
\usepackage{subcaption,cleveref}

\usepackage{float}

\usepackage{amsmath}
\usepackage{amssymb}
\usepackage{amsthm}

\usepackage{indentfirst}
\newtheorem{thm}{Theorem}[section]

\newtheorem{cor}[thm]{Corollary}
\newtheorem{lem}[thm]{Lemma}

\title[PERIODIC GRAPHS WITH REDUCIBLE DISPERSION POLYNOMIALS]{PERIODIC GRAPH OPERATORS WITH REDUCIBLE DISPERSION POLYNOMIALS FOR ALL POTENTIALS}
\author[Diantong Li]{Diantong Li}
\address{Institute of Applied Mathematics, Academy of Mathematics and Systems Science, Chinese Academy of Sciences, Beijing 100190, P.R.China}
\email{lidiantong25@mails.ucas.ac.cn}
\begin{document}

\begin{abstract}
  We give a complete characterization of periodic graph operators whose dispersion polynomials are reducible for every potential. Consequently, the reducibility dichotomy for parameter-dependent Laurent polynomials implies that, for every other periodic graph operator, the dispersion polynomial, and hence the Bloch variety, is irreducible for generic potentials.
  
  Our proof proceeds through two reductions. We first use the monodromy of the roots of the dispersion polynomial with respect to the potential parameters to reduce the problem to the splitting case. We then show that the splitting property passes to the principal submatrices of the Floquet matrix, reducing the problem further to the case that the fundamental domain contains two vertices.
\end{abstract}
\maketitle
\section{Introduction}
Bloch varieties are central objects in the spectral theory of periodic operators, and their irreducibility is closely related to spectral and inverse spectral problems; see \cite{Kuc16,SS26} for general background. Early work in the continuous setting includes the compactification approach of Bättig, Knörrer, and Trubowitz \cite{BKT91}. In the discrete setting, Liu \cite{Liu22} established irreducibility results for Bloch and Fermi varieties of discrete periodic Schr\"{o}dinger operators on $\mathbb Z^d$. These ideas were later extended to a broad class of $\mathbb Z^d$-periodic graphs; see \cite{FLM22,FLG25} for related results on the irreducibility of Bloch varieties for periodic graph operators. More recently, Faust and Liu \cite{faust2026generic} developed a different approach to study the generic irreducibility of Bloch varieties for general periodic graph operators.

This paper is motivated by the problems considered in \cite{faust2026generic} and \cite{faust2025absence}. In \cite{faust2026generic}, both edge weights and potential are treated as variables, and the authors prove that, under certain connectivity assumptions, the dispersion polynomial and the Bloch variety are generically irreducible. Their approach relies on a strong dichotomy that reduces the problem to a class of graphs with a simple structure.

For a periodic graph with fixed nonzero edge weights, one still has a dichotomy with respect to the potential: either the dispersion polynomial is reducible for every potential, or it is irreducible for generic potentials. However, for a given graph, this dichotomy does not provide any information about which of these two alternatives occurs. In \cite{faust2025absence}, the authors prove the absence of flat bands for generic potentials using a perturbation argument. This naturally leads to the question of whether the generic irreducibility result of \cite{faust2026generic} continues to hold when the edge weights are fixed and only the potential is allowed to vary. The answer is negative in general, as in \cite{faust2026generic} the authors construct a $\mathbb Z$-periodic graph with two vertices in a fundamental domain such that the dispersion polynomial is reducible for every potential. Surprisingly, our results show that every such counterexample must have essentially the same graph structure. In this paper, we give a complete characterization of the graph structure and edge weights for which the dispersion polynomial is reducible for every potential.

Let $\mathcal{L}_V$ be a $\mathbb Z^d$-periodic graph operator with periodic potential $V$. In this paper, the potentials are considered as variables. Consequently, after applying the Floquet transform, we obtain a finite-dimensional matrix family depending simultaneously on the Floquet variables and the potentials, which we denote by $L(z,V)$. 
Here
\[
z=(z_1,\ldots,z_d)\in(\mathbb C^\ast)^d
\]
is the complexified Floquet variable associated with the quasimomentum $k=(k_1,\ldots,k_d)$ through
\[
z_j=e^{2\pi i k_j},\qquad j=1,\ldots,d.
\]
The entries of $L(z,V)$ are Laurent polynomials in $z$, while the dependence on $V$ appears only on the diagonal. The spectral problem for the infinite periodic operator is encoded by a finite-dimensional matrix whose entries depend algebraically on both the Floquet variables and the potential parameters.

Accordingly, we introduce the dispersion polynomial
\[
D(z,V,\lambda)
=
\det(\lambda I - L(z,V)).
\]
When the potential $V$ is fixed, the corresponding Bloch variety is
\[
\mathcal B_V
=
\left\{
(z,\lambda)\in(\mathbb C^\ast)^d\times\mathbb C:
D(z,V,\lambda)=0
\right\}.
\]

For a $\mathbb{Z}^d$-periodic graph $\Gamma$, we choose a fundamental domain $W = \{u_1,...,u_N\}$. Define a finite graph $S(\Gamma/\mathbb{Z}^d)$ with vertex set $W$ and edge set
\[\mathcal{E}(S(\Gamma/\mathbb{Z}^d)) = \{ (u_i,u_j)\,|\,(u_i,u_j+a) \in \mathcal{E}(\Gamma)\, \mbox{for some} \, a \in \mathbb{Z}^d, i\neq j \}.\] 
We call $S(\Gamma/\mathbb{Z}^d)$ the quotient graph of $\Gamma$. For any $i,j \in \{1,...,N\}$, we define the set 
\[C_{ij}= \{a \in \mathbb{Z}^d \, | \,(u_i,u_j + a) \in \mathcal{E} \}.\]
For convenience, we can assume that the graph has no self-loop (the reason will be explained later), i.e., $0 \notin C_{ii}$. 
If the quotient graph is disconnected, then $L(z,V)$ can be written in block diagonal form, so the dispersion polynomial is always reducible. When $W$ only contains one vertex, $D(z,V,\lambda)$ is linear in $\lambda$, thus it is always irreducible. Therefore, in the following, we always assume that the quotient graph is connected and $|W| \geq 2$.

Now we are ready to state our main result:
\begin{thm}\label{main1}
$D(z,V,\lambda)$ is reducible for every $V \in \mathbb{C}^N$ if and only if the following conditions hold:
\begin{enumerate}
    \item[] $\mathrm{I}$. There exists a finite set $C \subset \mathbb{Z}^d$ such that $C_{ii} = C$ for every $i = 1,...,N$. Moreover, for any fixed $a \in C$, the edge weight of $(u_i,u_i+a)$ is independent of $i$. 
    \item[] $\mathrm{II}$. For every distinct $i,j \in \{1,...,N\}$, $|C_{ij}| \leq 1$.
    \item[] $\mathrm{III}$. Under condition $\mathrm{II}$, for any cycle in the quotient graph:
\[\mathcal{P} : u_{i_1} \rightarrow u_{i_2}\rightarrow \cdots \rightarrow u_{i_m} \rightarrow u_{i_1},\]
we have 
\begin{align}\label{quasi}
   \displaystyle\sum_{j = 1}^{m}\alpha_{i_ji_{j+1}} = 0, 
\end{align}
where $\alpha_{i_ji_{j+1}}$ is the unique element in $C_{i_ji_{j+1}}$, and the indices are taken modulo $m$.
\end{enumerate}
If any of the above conditions fails, $D(z,V,\lambda)$ is irreducible for generic $V \in \mathbb{C}^N$.
\end{thm}
As a corollary, we prove the corresponding result for Bloch varieties.
\begin{cor}\label{Bloch}
The Bloch variety $\mathcal{B}_V$ is reducible for generic $V \in \mathbb{C}^N$ if and only if the three conditions in Theorem $\ref{main1}$ hold. If any of these conditions fails, the Bloch variety is irreducible for generic $V \in \mathbb{C}^N$.
\end{cor}
By the strong dichotomy established in \cite{faust2026generic}, we only need to prove that $D(z,V,\lambda)$ is reducible for every potential if and only if conditions $\mathrm{I}$, $\mathrm{II}$ and $\mathrm{III}$ hold. The sufficiency is relatively easy, which will be treated in Section $\ref{sim}$. The following is an outline of the proof for the necessity:

The first step is a reduction from reducibility to splitting. More precisely, we prove that if $D(z,V,\lambda)$ is reducible in 
\[\mathbb C[z^{\pm}][\lambda]=
\mathbb C[z_1^{\pm1},\ldots,z_d^{\pm1}][\lambda]\]
for every \(V\in\mathbb C^N\), then it in fact splits over \(\mathbb C[z^{\pm}]\) for every \(V\in\mathbb C^N\).
The key point is to study the monodromy of the roots of $D(z,V,\lambda)$ with respect to the potential parameter $V$ (see Subsection $\ref{analytic}$ for a brief introduction). For generic $(z^0,V^0)$, the polynomial $D(z^0,V^0,\lambda)$ has only simple roots, which can locally be represented by holomorphic branches $\lambda_1(z,V),\ldots,\lambda_N(z,V)$.
Analytic continuation along loops in the potential parameter space permutes these branches and gives rise to a monodromy group. In Section $\ref{mono}$ we will prove that for generic $z^0$, this monodromy group is the full symmetric group $S_N$. Suppose that $F(z,\lambda)$ is a monic factor of $D(z,V^0,\lambda)$ whose local roots near $z^0$ are indexed by a subset $I\subset\{1,\ldots,N\}$. For each permutation $\sigma\in S_N$, choose a loop in the potential parameter space whose monodromy is $\sigma$. In Section $\ref{redutosplit}$, we show that the analytic continuation of the roots along this loop can give another monic factor $F_\sigma(z,\lambda)$ of $D(z,V^0,\lambda)$, whose local roots are indexed by $\sigma(I)$. Since the monodromy group is the symmetric group $S_N$, by choosing several suitable permutations and taking greatest common divisors of the corresponding factors, we can obtain a monic factor that is linear in $\lambda$, and its unique root  near $z^0$ is $\lambda_j(z,V^0)$, where $j \in \{1,...,N\}$ is arbitrary. Therefore, $\lambda_j(z,V^0) \in \mathbb{C}[z^{\pm}]$ and $D(z,V^0,\lambda)$ splits over $\mathbb{C}[z^{\pm}]$. The above argument is valid for generic $V^0 \in \mathbb{C}^N$.By an analogue of the strong dichotomy in \cite{faust2026generic}, we deduce that $D(z,V,\lambda)$ splits over $\mathbb{C}[z^{\pm}]$ for every $V \in \mathbb{C}^N$.

The second part is an inheritance principle for the splitting property. In Section $\ref{splitcase}$, we show that if $D(z,V,\lambda)$ splits over $\mathbb{C}[z^{\pm}]$ for every $V \in \mathbb{C}^N$, then for any $(N-1) \times(N-1)$ principal submatrix $L_{N-1}(z,V')$ obtained by deleting the $i$-th row and column of $L(z,V)$, its characteristic polynomial splits over $\mathbb{C}[z^{\pm}]$ for every $V' = (v_1,...,v_{i-1},v_{i+1},...,v_{N}) \in \mathbb{C}^{N-1}$. In other words, the splitting property is preserved after deleting a vertex from the quotient graph. By repeatedly applying this result, we reduce the problem to the case $N=2$, which is quite easy to solve. Applying the result for $N=2$ to all $2\times2$ principal submatrices of $L(z,V)$, we deduce that the conditions $\mathrm{I}$ and $\mathrm{II}$ hold. Once the first two conditions are established, the condition $\mathrm{III}$ follows relatively easily.

The rest of the paper is arranged as follows. In Section $\ref{bas}$, we introduce the setting of periodic graph operator and the concept of monodromy group. In Section $\ref{sim}$, we prove the sufficiency and establish the necessity in the case of $N=2$. In Section $\ref{pre}$, we provide some useful lemmas. In Sections $\ref{mono}$ and $\ref{redutosplit}$, we reduce Theorem $\ref{main1}$ to the splitting case. In Section $\ref{splitcase}$, we prove Theorem $\ref{main1}$  under the splitting assumption.
\section{Basic}\label{bas}
\subsection{Periodic graph and Periodic graph operator}
A locally finite graph $\Gamma = (\mathcal{V},\mathcal{E)}$ is said to be $\mathbb{Z}^d$ periodic if there is a $\mathbb{Z}^d$-action on $\Gamma$ with the following properties:
\begin{enumerate}
    \item The action is free, which means that for any non-zero element in $\mathbb{Z}^d$, its action does not fix any vertex.
    \item If $(u,v) \in \mathcal{E}$, then $(u+a,v+a) \in \mathcal{E}$ for any $a \in \mathbb{Z}^d$.
    \item The action is cocompact, i.e., $|\mathcal{V}/\mathbb{Z}^d| < \infty$.
\end{enumerate}

For any $\mathbb{Z}^d$-periodic graph $\Gamma$, we can choose a finite set $W \subset \mathcal{V}(\Gamma)$ such that 
\[ \mathcal{V}(\Gamma) = \bigcup_{a \in \mathbb{Z}^d} W +a,\]
and $|W|= |\mathcal{V}(\Gamma)/\mathbb{Z}^d|$. We call it a fundamental domain of $\Gamma$. Replacing each vertex in the fundamental domain by an arbitrary vertex from the same equivalence class will give another fundamental domain.

A $\mathbb{Z}^d$-periodic edge weight $E:\mathcal{E}(\Gamma) \rightarrow\mathbb{C}$ is a function on $\mathcal{E}(\Gamma)$ that satisfies for every $(u,v) \in \mathcal{E}(\Gamma)$, $a \in \mathbb{Z}^d$,
\[ E((u,v)) = E((v,u)), \quad E((u,v)) =  E((u+a,v+a)).\]
In this paper, we assume that $E((u,v))\neq 0$ for every $(u,v) \in \mathcal{E}(\Gamma)$. Otherwise, we can simply delete this edge.

A $\mathbb{Z}^d$-periodic potential $V:\mathcal{V}(\Gamma) \rightarrow \mathbb{C}$ is a function on $\mathcal{V}(\Gamma)$ that satisfies $V(u + a) = V(u)$ for every $u \in \mathcal{V}(\Gamma)$, $a \in \mathbb{Z}^d$.

For a $\mathbb{Z}^d$-periodic graph with the edge weight function $E$ and the potential function $V$, we can define the $\mathbb{Z}^d$-periodic graph operator $\mathcal{L}_V$ on $\ell^2(\mathcal{V}(\Gamma))$ as follows:
\[ \mathcal{L}_V(f)(u) = V(u)f(u) + \displaystyle\sum_{(u,v) \in \mathcal{E}(\Gamma)}^{}E((u,v))f(v).\]
Note that for any self-loop, the corresponding terms can be absorbed into potentials, so we may assume that $\Gamma$ has no self-loops.
\subsection{Floquet theory}
Fix a fundamental domain $W = \{u_1,...,u_N\}$. For a finitely supported function $f$ on $\mathcal{V}(\Gamma)$, define its Floquet transform by
\[\widehat{f}(z,u_i)
=
\sum_{a\in\mathbb{Z}^d} f(u_i+a)z^{-a},
\qquad
u_i\in W,\ z\in\mathbb{T}^d,\]
where $z^a=z_1^{a_1}\cdots z_d^{a_d}$ for $a=(a_1,\dots,a_d)\in\mathbb{Z}^d$. Note that if we define a function on $\mathbb{Z}^d$ by $a \rightarrow f(u_i +a)$, then $\widehat{f}(z,u_i)$ is exactly the Fourier transform of this function. Applying Parseval's identity for each $i \in \{1,...,N\}$ and taking a sum, we obtain
\[\sum_{u\in W}
\int_{\mathbb{T}^d}
|\widehat{f}(z,u)|^2\,dz
=
\sum_{v\in\mathcal{V}(\Gamma)}
|f(v)|^2.\]
For any $z \in \mathbb{T}^d$, consider $(\widehat{f}(z,u_1),\ldots ,\widehat{f}(z,u_N))^{T}$ as a vector in $\mathbb{C}^{N}$. The map $\mathcal{F}:f\mapsto\widehat{f}$ extends to a unitary operator
\[
\mathcal{F}:
\ell^2(\mathcal{V}(\Gamma))
\longrightarrow
L^2(\mathbb{T}^d,\mathbb{C}^{N}).
\]
Therefore, the operator
\[ \mathcal{F}\mathcal{L_V}\mathcal{F}^*: L^2(\mathbb{T}^d,\mathbb{C}^{N}) \rightarrow L^2(\mathbb{T}^d,\mathbb{C}^{N})\]
is unitarily equivalent to $\mathcal{L}_V$. To describe this transformed operator, it is convenient to introduce
\begin{align}\label{laur}
  l_{ij}(z)
:=
\sum_{\substack{
a\in\mathbb{Z}^d\\
(u_i,u_j+a)\in\mathcal{E}(\Gamma)
}}
E((u_i,u_j+a))z^a  
\end{align}
for $u_i,u_j \in W$, $z \in (\mathbb{C}^*)^d$. Since $\Gamma$ is locally finite, the sum above has finitely many terms. Hence $l_{ij}(z)$ is a Laurent polynomial in $z_1,\dots,z_d$. By the symmetry of the edge weight function, 
\[ l_{ij}(z) = l_{ji}(z^{-1} ).\]

Now we consider the ``fiber'' of $\mathcal{F}\mathcal{L}_V\mathcal{F}^*$, i.e., we fix a $z \in \mathbb{T}^d$  and consider it as an operator from $\mathbb{C}^N$ to $\mathbb{C}^N$. By direct calculation, we have
\begin{align}\label{Flo}
\mathcal{F}\mathcal{L}_V\mathcal{F}^*(\widehat{f})(z,u_i)
=
V(u_i)\widehat{f}(z,u_i)
+ \sum_{j = 1}^{N}
l_{ij}(z)\widehat{f}(z,u_j).
\end{align}
For each $z \in \mathbb{T}^d$, define an $N\times N$ matrix $L(z,V)$ by
\begin{align}\label{Fmatrix}
L(z,V)_{i,j} = \delta_{ij}V(u_i) + l_{ij}(z),    
\end{align}
where $\delta$ is the Kronecker delta function. We call it the Floquet matrix of the operator $\mathcal{L}_V$. By $\eqref{Flo}$, $\mathcal{F}\mathcal{L}_V\mathcal{F}^*$ is the direct integral of $L(z,V)$ over $\mathbb{T}^d$, i.e., if $\widehat{f} \in L^2(\mathbb{T}^d,\mathbb{C}^{N})$ and $\mathcal{F}\mathcal{L}_V\mathcal{F}^*(\widehat{f})  = \widehat{g}$, then 
\[ (\widehat{g}(z,u_1),\cdots,\widehat{g}(z,u_N))^{T} = L(z,V)(\widehat{f}(z,u_1),\cdots,\widehat{f}(z,u_N))^{T}\]

From $\eqref{laur}$ and $\eqref{Fmatrix}$, we can reformulate the conditions $\mathrm{I}$ and $\mathrm{II}$ as follows:
\begin{enumerate}
    \item[]$\mathrm{I}$.  For any distinct $i,j \in \{1,...,N\}$, $l_{ii}(z) = l_{jj}(z)$
    \item[] $\mathrm{II}$. For any distinct $i,j \in \{1,...,N\}$, $l_{ij}(z)$ is either identically zero or a monomial. 
\end{enumerate}
\subsection{Discriminant}
Let $p(\lambda)=a_n\lambda^n+\cdots+a_0$, where $a_n\ne0$. Its discriminant is defined by
\begin{align}\label{disc}
\operatorname{Disc}_{\lambda}(p)
=\frac{(-1)^{n(n-1)/2}}{a_n}\det S(p,p'),    
\end{align}
where $S(p,p')$ is the $(2n-1)\times(2n-1)$ Sylvester matrix
\[
S(p,p')=\left(
\begin{smallmatrix}
a_n & a_{n-1} & \cdots & a_1 & a_0 & 0 & \cdots & 0\\
0 & a_n & a_{n-1} & \cdots & a_1 & a_0 & \ddots & \vdots\\
\vdots & \ddots & \ddots & \ddots & & \ddots & \ddots & 0\\
0 & \cdots & 0 & a_n & a_{n-1} & \cdots & a_1 & a_0\\[4pt]
na_n & (n-1)a_{n-1} & \cdots & a_1 & 0 & 0 & \cdots & 0\\
0 & na_n & (n-1)a_{n-1} & \cdots & a_1 & 0 & \ddots & \vdots\\
\vdots & \ddots & \ddots & \ddots & & \ddots & \ddots & 0\\
0 & \cdots & 0 & 0 & na_n & (n-1)a_{n-1} & \cdots & a_1 
\end{smallmatrix}\right).
\]
The discriminant vanishes if and only if $p$ has a multiple root. We set 
\[ \Delta(z,V) = \operatorname{Disc}_\lambda D(z,V,\lambda).\]
From $\eqref{disc}$, $\Delta(z,V)$ is a polynomial in $V$ and a Laurent polynomial in $z$.
\subsection{Analytic continuation of roots and monodromy group}\label{analytic}
A classical theme in the theory of algebraic equations is the correspondence between the monodromy of polynomial roots and the associated Galois group over a function field. This provides a natural bridge between analytic continuation and algebraic properties of polynomial equations. Related studies can be found in \cite{khovanskii2019representability,Harris1979,Ritt1922}.
Here, we focus on explaining how this concept applies to the setting of the present paper. For any fixed $z^0 \in (\mathbb{C}^*)^d$, $\Delta(z^0,V) \neq 0$ on a non-empty Zariski open set $\mathcal{U}$. We choose a $V^0 \in \mathcal{U}$ as the base point. Near $(z^0,V^0)$, the roots of $D(z,V,\lambda)$ can be represented as jointly holomorphic functions $\lambda_1(z,V),....,\lambda_N(z,V)$. Suppose that $\gamma :[0,1] \rightarrow \mathcal{U}$ is a loop based at $V^0$. Then $\lambda_j(z,V)$ can be analytically continued along $\gamma$.  More specifically, we can choose a neighborhood $\Omega$ of $z^0$, a subdivision
\[ 0 = t_0 < t_1<\cdots <t_r = 1,\]
open balls $B_q$, $q = 1,..,r$ with $\gamma|_{[t_{q-1},t_{q}]} \subset B_q$, and
\[\lambda_j^{(q)}(z,V)  \in \mathcal{O}(\Omega\times B_q), \,q = 1,...,r, j=1,...,N\]
such that:
\begin{enumerate}
    \item $\lambda_1^{(q)}(z,V),...,\lambda_N^{(q)}(z,V)$ are the distinct roots of $D(z,V,\lambda)$ for every $(z,V) \in \Omega \times B_q$.
    \item $\lambda_j^{(1)}(z,V) = \lambda_j(z,V)$ for every $V \in B_1$
    \item $\lambda_j^{(q)}(z,V)|_{\Omega \times (B_q\cap B_{q+1})} = \lambda_j^{(q+1)}(z,V)|_{\Omega \times (B_q\cap B_{q+1})}$, $q = 1,...,r-1$. 
\end{enumerate}
For a fixed $i \in \{1,...,N\}$, $\lambda_i^{(r)}(z,V)$ may not coincide with $\lambda_i(z,V)$ on $\Omega \times( B_r\cap  B_{1})$. However, since $\lambda_i^{(r)}(z,V)$ is also a root of $D(z,V,\lambda)$, by condition (1), it must coincide with some $\lambda_j(z,V)$ on $\Omega \times (B_r\cap B_{1})$. Note that if $\lambda_{i_1}^{(r)}(z,V)$ and $\lambda_{i_2}^{(r)}(z,V)$ coincide with the same $\lambda_j(z,V)$, then by condition (3) and the unique continuation theorem, $\lambda_{i_1}(z,V)$ must coincide with $\lambda_{i_2}(z,V)$ on $\Omega \times B_1$, which contradicts condition (1). Therefore, every loop in $\mathcal{U}$ induces a permutation in $S_N$.  It is easy to prove that if $\gamma_1$ and $\gamma_2$ are homotopic, then they induce the same permutation. Thus we can define a map 
\[ \varphi_{z^0} : \pi_1(U,V^0) \rightarrow S_N.\]
The image of $\varphi_{z^0}$ is called the \textbf{monodromy group} of $D(z^0,V,\lambda)$ (the choice of the base point will not change the monodromy group up to a conjugacy). When $z^0$ is fixed, $D(z^0,V,\lambda) \in \mathbb{C}[v_1,...,v_N][\lambda]$. From the well-known monodromy-Galois correspondence, 
\[ \mathrm{Im}(\varphi_{z^0}) \cong \mathrm{Gal}(D(z^0,V,\lambda)/\mathbb{C}(v_1,...,v_N)).\]
An elementary proof can be found in \cite{khovanskii2019representability}. A natural question is whether the monodromy group changes with the choice of $z^0$. We can prove that for generic choices of $z^0$, the monodromy group is the same (up to a conjugacy). This fact is not easy to obtain in the general case, but the special structure of $D(z,V,\lambda)$ simplifies the situation. In Section $\ref{mono}$, we will prove that for generic $z^0$, the monodromy group of $D(z^0,V,\lambda)$ is $S_N$.

For convenience, we collect some notation that will be used in the proofs.
\begin{enumerate}
    \item $R = \mathbb{C}[z_1^{\pm},...,z_d^{\pm}]$ is the Laurent polynomial ring. $M_n(R)$ denotes the set of $n\times n$ matrices with entries in $R$.
    \item $V = (v_1,...,v_N) \in \mathbb{C}^N$ is the potential function. 
    \item For $\alpha = (\alpha_1,...,\alpha_d) \in \mathbb{Z}^d$, $z^\alpha = z_1^{\alpha_1} \cdots z_d^{\alpha_d}$.
    \item Under condition $\mathrm{II}$, for any cycle
    \[ \mathcal{P} : u_{i_1} \rightarrow \cdots \rightarrow u_{i_m}\rightarrow u_{i_1}\]
    in the quotient graph, define 
    \[\mathrm{quasi}(\mathcal{P}) = \alpha_{i_1i_2} +\cdots +\alpha_{i_mi_1},\]
    where $\alpha_{i_ji_{j+1}}$ is the unique element in $C_{i_ji_{j+1}}$.
    \item For any (Laurent) polynomial $F(\xi)$, we use $[\xi^\beta]F$ to denote the coefficient of the term $\xi^\beta$ in $F$.
    \item For an open set $U \subset \mathbb{C}^n$, $\mathcal{O}(U)$ denotes the holomorphic function ring on $U$.
\end{enumerate}

\section{Simple cases}\label{sim}
In this section, we consider several simple cases whose conclusions provide useful insight into the proof of the general case. We first explain the meaning of ``generic'' used in this paper. A subset $Z \subset \mathbb{C}^N$ is called \textbf{Zariski closed} if there exist polynomials
\[ P_1,....,P_k \in \mathbb{C}[v_1,...,v_N]\]
such that 
\[ Z = \{V \in \mathbb{C}^N\, | \,P_1(V) = \cdots=P_k(V) = 0 \}.\]
A subset $U \subset \mathbb{C}^N$ is Zariski open if its complement is Zariski closed. We say a property holds for generic potentials if there exists a non-empty Zariski open set $U \subset \mathbb{C}^N$ such that this property holds for any $V \in U$.

From Theorem 2.1 in \cite{faust2026generic}, the set 
\[\mathcal{R} = \{V \in \mathbb{C}^N \, | \, D(z,V,\lambda)\, \mbox{is reducible in}\, R[\lambda] \} \]
is Zariski closed. Consequently, reducibility exhibits the following dichotomy with respect to the potential parameters: 
\begin{enumerate}
    \item $D(z,V,\lambda)$ is reducible for every $V \in \mathbb{C}^N$.
    \item $D(z,V,\lambda)$ is irreducible for generic $V \in \mathbb{C}^N$.
\end{enumerate}
Therefore, we only need to prove that $D(z,V,\lambda)$ is reducible for every $V \in \mathbb{C}^N$ if and only if the three conditions in Theorem $\ref{main1}$ hold. We begin with the sufficiency, which is comparatively easy to establish, and we can obtain a stronger result:
\begin{thm}
Assume that conditions $\mathrm{I}$, $\mathrm{II}$ and $\mathrm{III}$ hold. Then for any $V \in \mathbb{C}^N$, $D(z,V,\lambda)$ splits with respect to $\lambda$ over $R$
\end{thm}
\begin{proof}
Since the quotient graph is connected, for any vertex $u_j$, we can choose a path
\[\mathcal{P}_j: u_{1} = u_{i_1} \rightarrow\cdots\rightarrow u_{i_m} = u_{j} \]
and define 
\[ \tau_1 = 0,\quad \tau_j =  \displaystyle\sum_{k = 1}^{j-1}\alpha_{i_ki_{k+1}}, \quad 2 \leq j\leq m.\]
 From condition $\mathrm{III}$, the value of $\tau_j$ is independent of the path we choose. For any $i,j \in \{1,...,N\}$ such that $u_i$ and $u_j$ are adjacent, we can obtain a path from $u_1$ to $u_j$ by adjoining the edge $(u_i,u_j)$ to a path from $u_1$ to $u_i$, thus
\begin{align}\label{cancel}
  \alpha_{ij} = \tau_j - \tau_i.   
\end{align}
Set
\[ Q = \mathrm{diag}(z^{\tau_1},...,z^{\tau_N}), \quad 
\widetilde{L}(z,V) =  QL(z,V)Q^{-1}.\]
we have 
\[\widetilde{l}_{ij}(z) = z^{\tau_i - \tau_j}l_{ij}(z). \]
By $\eqref{cancel}$, $\widetilde{l}_{ij}(z)$ is a constant for any distinct $i,j$. Combined with condition $\mathrm{I}$, there exists $\phi(z) \in R$ such that 
\[ \widetilde{L}(z,V) = \phi(z) I + \mathrm{diag}(v_1,...,v_N) + H,\]
where $H$ is a constant matrix that depends on the edge weight. For any fixed $V \in \mathbb{C}^N$, assume that 
\[ \det(\lambda I- \mathrm{diag}(v_1,...,v_N)-H) = \displaystyle\prod_{j = 1}^{N}(\lambda-\lambda_j(V)),\]
then we have 
\begin{align}
    D(z,V,\lambda) = \det(\lambda I -\widetilde{L}(z,V)) = \displaystyle\prod_{j = 1}^{N}(\lambda-\phi(z) - \lambda_j(V)).
\end{align}
\end{proof}
From a graph-theoretic point of view, these three conditions mean that one can choose a suitable fundamental domain such that every edge between distinct equivalence classes lies within a single translate of it, while different translates are connected only through vertices in the same equivalence class. Moreover, the connection pattern within each equivalence class, together with the corresponding edge weights, is independent of the chosen vertex. In other words, we can construct a graph on $\mathbb Z^d$ with the edge set
\[ \mathcal{E}(\mathbb{Z}^d) = \{ (n,n+a) \,|\, n \in \mathbb{Z}^d, a \in \mathcal{A} \},\]
where $\mathcal{A}$ is a finite subset of $\mathbb{Z}^d$, such that $\Gamma$ can be represented as the Cartesian product of the quotient graph and this graph on $\mathbb Z^d$.

The proof of necessity is more complicated. We first treat the case $N=2$. As we shall see later, the general case can be reduced to this special case. Note that in the case of $N = 2$, condition $\mathrm{III}$ is satisfied automatically.
\begin{thm}\label{simple}
If $D(z,V,\lambda)$ is reducible for any $V \in \mathbb{C}^2$, then $l_{11}(z) = l_{22}(z)$, and there exists $\alpha \in \mathbb{Z}^d$, $c\in \mathbb{C}^*$, such that
\[l_{12}(z) = cz^\alpha, \quad l_{21}(z) = cz^{-\alpha}.\]
\end{thm}
\begin{proof}
 If $D(z,V,\lambda)$ is reducible for every potential, then for any $V = (v_1,v_2) \in \mathbb{C}^2$, there exist $\lambda_1(z,V),\lambda_{2}(z,V) \in R$ such that 
\[ (\lambda - l_{11}(z) - v_1)(\lambda - l_{22}(z) - v_2) - l_{12}(z)l_{12}(z^{-1}) = (\lambda - \lambda_{1}(z,V))(\lambda - \lambda_{2}(z,V)).\]
By analogy with the discriminant of an ordinary quadratic equation, we obtain
\[(l_{11}(z) - l_{22}(z) + v_1 - v_2)^2 + 4l_{12}(z)l_{21}(z) = (\lambda_{1}(z,V) -\lambda_{2}(z,V))^2. \]
Define $a(z) = l_{11}(z) - l_{22}(z)$, $b(z) =  4l_{12}(z)l_{12}(z^{-1})$, then for  any $v \in \mathbb{C}$, there exists $f_v \in R$ such that
\begin{align}\label{decmo}
    (f_v(z) + a(z) + v)(f_v(z) - a(z) - v) = p_v(z)q_v(z) = b(z).
\end{align}
Since the quotient graph is connected, $b$ is not identically zero. By the fact that $R$ is a unique factorization domain, there exist $p,q \in R$ such that $pq = b$ and for infinitely many $v$, there exist corresponding $\lambda_v \in \mathbb{C}^*$ and $m_v \in \mathbb{Z}^d$ such that
\[p_v = \lambda_v z^{m_v}p, \quad q_v = \lambda_v^{-1}z^{-m_v}q. \]
Notice that 
\begin{align}\label{minus}
p_v(z) - q_v(z) = 2(a(z)+v),   
\end{align}
comparison of the support on both sides implies that $m_v$ takes only finitely many values as $v$ varies. Hence, there exists an infinite set $\Lambda \subset \mathbb{C}$ and $m \in \mathbb{Z}^d$  such that
\begin{align}\label{expre}
 p_v = \lambda_vz^m p, \quad q_v = \lambda_v^{-1}z^{-m}q, \quad \forall\, v \in \Lambda.   
\end{align}
From \eqref{minus}, we deduce that for any distinct $v,v' \in \Lambda$, $\lambda_v \neq \lambda_{v'}$. Taking three distinct $v,v',v'' \in \Lambda$ and using $\eqref{expre}$, we have
\begin{align}
\begin{aligned}
(\lambda_{v}-\lambda_{v'})z^mp - (\lambda_{v}^{-1}-\lambda_{v'}^{-1})z^{-m}q = 2(v-v')
\\ (\lambda_{v}-\lambda_{v''})z^mp - (\lambda_{v}^{-1}-\lambda_{v''}^{-1})z^{-m}q = 2(v-v'')   
\end{aligned}
\end{align}
Combining the above two equations, we deduce that $z^mp$ and $z^{-m}q$ are constant, thus $b(z)$ and $a(z)$ are constant. Note that $l_{12}(z)l_{12}(z^{-1})$ is constant if and only if $l_{12}(z)$ is a monomial. Since $\Gamma$ has no self-loop, $l_{ii}(z)$ does not contain a constant term, it follows that $l_{11}(z) = l_{22}(z)$.
  
\end{proof}
\section{Preliminary lemmas}\label{pre}
In this section, we present some technical lemmas that are useful for subsequent proofs.
\begin{lem}\label{converge}
Suppose that 
\[ A = \begin{pmatrix}
 A_{n-1}&\alpha \\ 
\beta^{T} & a_{nn}
\end{pmatrix} \in M_n(\mathbb{C}),\]
and $A_{n-1}$ has simple eigenvalues $\lambda_1$,...,$\lambda_{n-1}$. Set 
\[B(u) = \mathrm{diag}(0,..,0,u) + A, \quad u \in \mathbb{C}.\]
For some sufficiently large $M>0$, there exist $n-1$ holomorphic functions $\lambda_1(u),...,\lambda_{n-1}(u)$ defined on $\{ u \in \mathbb{C}\, |\,|u|> M\}$, such that for any $j = 1,...,n-1$, $\lambda_j(u)$ is an eigenvalue of $B(u)$ and $\lambda_j(u) \rightarrow \lambda_j$ as $u \rightarrow \infty$. The remaining eigenvalue of $B(u)$ has the form $u + O(1)$.
\end{lem}
\begin{proof}
Expanding \(\det(\lambda I-B(u))\) along the last row, we obtain:
\[ \det(\lambda I - B(u)) = (\lambda - u - a_{nn})\displaystyle\prod_{i = 1}^{n-1}(\lambda-\lambda_i) + R(\lambda).\]
Set $w = u^{-1}$ and define 
\[ p(\lambda,w) = w\det(\lambda I -B(u)) = (\lambda w -1-wa_{nn})\displaystyle\prod_{i = 1}^{n-1}(\lambda-\lambda_i) + wR(\lambda),\]
For convenience, we extend the range of $w$ to $\mathbb C$. For $w \neq 0$, $p(\lambda,w) = 0$ if and only if $\lambda$ is an eigenvalue of $B(w^{-1})$. Since $\lambda_1,...,\lambda_{n-1}$ are distinct,
\[ \frac{\partial p}{\partial \lambda}(\lambda_j,0) \neq 0,\quad \forall \,j = 1,...,n-1.\]
By the holomorphic implicit function theorem, for any $j = 1,...,n-1$, there exists a holomorphic function $\rho_j(w)$ near $w = 0$ such that 
\[ p(\rho_j(w),w) = 0, \quad \rho_j(0) = \lambda_j.\]
Therefore, the function $\lambda_j(u)=\rho_j(u^{-1})$ satisfies the desired properties. The last conclusion follows directly from the Gershgorin circle theorem.

\end{proof}
The proof of the next lemma can be found in \cite{faust2026generic}.
\begin{lem}[\cite{faust2026generic}, Lemma 2.3]\label{supp}
Suppose that $A(z) \in M_n(R)$. For any $V  = (v_1,...,v_n) \in \mathbb{C}^n$, define $B(z,V) = \mathrm{diag}(v_1,...,v_n) + A(z)$ and $P(z,V,\lambda) = \det(\lambda I -B(z,V))$. There exists a finite set $\mathcal{S} \subset \mathbb{Z}^d$ such that if $P(z,V^0,\lambda)$ is reducible for some $V^0 \in \mathbb{C}^n$ and  
\[F(z,\lambda) = \lambda^m + \displaystyle\sum_{k = 0}^{m-1}f_k(z)\lambda^k\]
is a monic factor, then $\mathrm{supp}(f_k) \subset \mathcal{S} $.
\end{lem}

The following lemma is very similar to Proposition 2.4 in \cite{faust2026generic}.
\begin{lem}\label{closed}
The set 
\[ \widetilde{\mathcal{R}} = \{V \in \mathbb{C}^N \, | \, D(z,V,\lambda)\, \mbox{splits over}\, R \}\]
is Zariski closed,
\end{lem}
\begin{proof}
    By Lemma $\ref{supp}$, there exists a finite set $\mathcal{S}\subset\mathbb Z^d$ such that if $h(z) \in R$ and $\lambda-h(z)$ divides $D(z,V,\lambda)$, then $\operatorname{supp}(h) \subset \mathcal{S}$. Set
\[
X=\mathbb C\lambda\oplus
\operatorname{span}_{\mathbb C}\{z^\alpha \,| \, \alpha\in \mathcal{S}\} \subset R[\lambda],
\]
and choose a finite-dimensional subspace $Y\subset R[\lambda]$ such that $D(z,V,\lambda)\in Y$ for every $V\in\mathbb C^N$ and $F_1\cdots F_N\in Y$ for every $F_1,\ldots,F_N\in X$.

Multiplication induces a morphism
\[
\mu:(\mathbb P(X))^N\longrightarrow\mathbb P(Y),
\qquad
([F_1],\ldots,[F_N])\longmapsto[F_1\cdots F_N].
\]
Its image is Zariski closed because its domain is projective. Moreover,
\[
\rho:\mathbb C^N\longrightarrow\mathbb P(Y),
\qquad V\longmapsto[D(z,V,\lambda)]
\]
is a regular map, since the coefficients of $D$ depend polynomially on $V$ and its leading coefficient in $\lambda$ is $1$.

We claim that
\[
\widetilde{\mathcal{R}}=\rho^{-1}(\operatorname{im}\mu).
\]
One inclusion follows from Lemma $\ref{supp}$. Conversely, if $[D]\in\operatorname{im}\mu$, then $\deg_\lambda D=N$ forces every factor $F_j\in X$ to have a nonzero coefficient of $\lambda$. Normalizing these factors to be monic yields a complete splitting of $D$ over $R$. Thus $\widetilde{\mathcal{R}}$ is Zariski closed.
\end{proof}
\section{Monodromy group}\label{mono}
The aim of this section and the next section is to reduce the general case to the splitting case, i.e., we will prove that if $D(z,V,\lambda)$ is reducible for every $V \in \mathbb{C}^N$, then $D(z,V,\lambda)$ splits over $R$ for every $V \in \mathbb{C}^N$. 

In this and the next section, we assume that $z^0 \in (\mathbb{C}^*)^d$ satisfies $l_{ij}(z^0)\neq 0$ for every distinct $i,j \in\{1,...,N\}$ with $l_{ij}(z)\not\equiv 0$. Such points are generic in $(\mathbb{C}^*)^d$. We denote $D(z^0,V,\lambda)$ by $\Phi(V,\lambda)$ and $\varphi_{z^0}$ by $\varphi$. We will prove that the connectivity of the quotient graph implies the irreducibility of \(\Phi\), which in turn imposes restrictions on \(\operatorname{Im}(\varphi)\).
\begin{lem}\label{irre}
Regard $y_1,\ldots,y_N$ as variables, and define the matrix $H(y_1,...,y_N) = (h_{ij})_{1\leq i,j\leq N}$ by
\[ h_{ij} = \left\{\begin{matrix}
y_i \quad \quad\mbox{if} \, \, i= j\\ 
-l_{ij}(z^0) \quad \quad\mbox{if} \, \, i\neq j
\end{matrix}\right.\]
Set $\Psi(y_1,...,y_N) = \det(H(y_1,...,y_N))$. Then $\Psi$ is irreducible in $\mathbb{C}[y_1,...,y_N]$.
\end{lem}
\begin{proof}
Assume, by contradiction, that $\Psi$ is reducible. There exist $f,g \in \mathbb{C}[y_1,..,y_N]$ such that 
\[ \Psi = fg.\]
As $\deg_{y_i}(\Psi) = 1$, each $y_i$  appears in exactly one of the factors $f$ and $g$. Moreover, we have
\[[y_1\cdots y_N]\Psi = 1,\]
This gives a partition
\[ \{1,2,...,N\} = I\sqcup J, \quad I,J \neq \varnothing \]
such that $f \in \mathbb{C}[y_i, i\in I]$, $g \in \mathbb{C}[y_j,j\in J]$. Moreover, after multiplying $f$ and $g$ by suitable constants, we may assume that
\begin{align}\label{nozero}
[\displaystyle\prod_{k\in I}^{}y_k]f = 1,\quad [\displaystyle\prod_{k\in J}^{}y_k]g = 1    
\end{align}

Since the quotient graph is connected, there exist $i \in I$ and $j \in J$ such that $l_{ij}(z^0) \neq 0$, thus 
\begin{align}\label{nozero2}
 [\displaystyle\prod_{k \neq i,j}^{}y_k ]\Psi =-l_{ij}(z^0)l_{ji}(z^0) \neq 0.   
\end{align}
However, we have 
\[ [\displaystyle\prod_{k \neq i}^{}y_k]\Psi = [\displaystyle\prod_{k \neq j}^{}y_k]\Psi =0.\]
Combined with $\eqref{nozero}$, we deduce that 
\[[\displaystyle\prod_{k \in I,k \neq i}^{}y_k]f = 0, \quad [\displaystyle\prod_{k \in J,k \neq j}^{}y_k]g = 0,\]
which contradicts $\eqref{nozero2}$.
\end{proof}
The Lemma $\ref{irre}$ implies that $\Phi$ is irreducible in $\mathbb{C}[v_1,...,v_N][\lambda]$. Indeed, if $\Phi$ were reducible in $\mathbb{C}[v_1,\ldots,v_N][\lambda]$, then the change of variables
\[
y_i=\lambda-v_i-l_{ii}(z^0) ,\quad i = 1,...,N
\] 
would imply that $\Psi$ is reducible in $\mathbb{C}[y_1,...,y_N]$, which is a contradiction. Since $\Phi$ is monic, by Gauss Lemma, $\Phi$ is irreducible in $\mathbb{C}(v_1,...,v_N)[\lambda]$. Therefore, $\mathrm{Gal}(\Phi/\mathbb{C}(v_1,...,v_N))$ is a transitive subgroup of $S_N$.

The following theorem is related to the work in \cite{cohen1990galois}.
\begin{thm}\label{symm}
$\mathrm{Im}(\varphi) = S_N$.
\end{thm}
\begin{proof}
We argue by induction on $N$. The conclusion is obvious when $N=1$. Assume that the conclusion holds for $N-1$. Relabel the vertices, if necessary, such that the subgraph of the quotient graph induced by the first $N-1$ vertices is also connected. We set 
\[ L(z,V) =  \begin{pmatrix}
 L_{N-1}(z,V')& \alpha(z)\\ 
 \beta^T(z)& l_{NN}(z) +v_N
\end{pmatrix}.\]
Then for any $\sigma \in S_{N-1}$, there exists a neighborhood $\Omega$ of $z^0$ and a loop $\gamma$ in $\mathbb{C}^{N-1}$ such that for $z \in \Omega$, $V' = (v_1,..,v_{N-1}) \in \gamma$, $L_{N-1}(z,V')$ has only simple eigenvalues, and the permutation induced by $\gamma$ is $\sigma$. There exists $\varepsilon >0$ such that for any distinct $i,j \in \{1,...,N-1\}$,
\[|\lambda_i(z,V') - \lambda_j(z,V')| \geq \varepsilon,\quad \forall \,z\in \Omega, V' \in \gamma.\]
Consider the loop $\gamma_M(t) = (\gamma(t),M)$ in $\mathbb{C}^N$. From the argument in Lemma $\ref{converge}$, when $|M|$ is sufficiently large, for any $V = (V',M) \in \gamma_M$ and $z \in \Omega$, $L(z,V)$ has $N-1$ simple eigenvalues $\lambda_i(V)$, $i = 1,2,..,N-1$ that satisfy
\[|\lambda_i(z,V)-\lambda_j(z,V)| \geq \frac{\varepsilon}{2}, \quad |\lambda_i(V)| \leq \frac{|M|}{2}, \]
and an eigenvalue $\lambda_N(z,V)$ with 
\[|\lambda_N(z,V)| > \frac{|M|}{2}.\]
Consequently, the permutation induced by $\gamma_M$ must be 
\[ i \mapsto \sigma(i) ,\quad 1\leq i\leq N-1,\quad N\mapsto  N.\]
Therefore, the stabilizer subgroup 
\[ \operatorname{Stab}(N) = \{\eta \in \operatorname{Im}(\varphi)| \, \eta(N) = N \}\]
is isomorphic to $S_{N-1}$. Moreover, from Lemma $\ref{irre}$, $\operatorname{Gal}(\Phi/\mathbb{C}(v_1,..,v_N))$ (hence $\operatorname{Im}(\varphi)$) is a transitive subgroup of $S_N$, thus
\[ |\operatorname{Im}(\varphi)| \geq N|S_{N-1}| = N!.\]
We deduce that $\operatorname{Im}(\varphi) = S_N$.
\end{proof}
\section{Reduction theorem}\label{redutosplit}
In this section, we assume that $D(z,V,\lambda)$ is reducible for every $V \in \mathbb{C}^N$. Define
\[ \mathcal{U} = \{V \in \mathbb{C}^N\, | \, \Delta(z^0,V) \neq 0 \}\]
 and set $V^0 \in \mathcal{U}$. Then there exists a sufficiently small neighborhood $\Omega \times U$ of $(z^0,V^0)$, together with jointly holomorphic functions
 \[
\lambda_1(z,V),\ldots,\lambda_N(z,V)\in\mathcal O(\Omega\times U),
\]
such that for every $(z,V)\in\Omega\times U$, $\lambda_1(z,V),\ldots,\lambda_N(z,V)$ are precisely the roots of $D(z,V,\lambda)$. Moreover, for any distinct $i,j \in\{1,...,N\}$ and $(z,V) \in \Omega \times U$, we have
\begin{align}\label{diff}
 \lambda_i(z,V) \neq \lambda_j(z,V).   
\end{align}
\begin{lem}\label{holomorphic}
There exists $\mathcal{J} \subsetneq  \{1,...,N\}$ such that 
\[ F(z,V,\lambda) = \displaystyle\prod_{j\in \mathcal{J}}^{}(\lambda-\lambda_j(z,V))\]
is a factor of $D(z,V,\lambda)$ for every $V \in U$.
\end{lem}
\begin{proof}
Fix an arbitrary $V^1 \in U$. By $\eqref{diff}$, any monic proper factor of $D(z,V^1,\lambda)$ can be written as
\begin{align}\label{factor}
 F_\mathcal{J}(z,V^1,\lambda) = \displaystyle\prod_{j\in \mathcal{J}}^{}(\lambda-\lambda_j(z,V^1))   
\end{align}
for some $\mathcal{J} \subsetneq \{1,...,N\}$. For convenience, we assume that $\mathcal{J} = \{1,...,m\}$. Set
\begin{align}\label{symmcoe}
  f_{\mathcal{J},k}(z,V) = (-1)^{m-k}e_{m-k}(\lambda_1(z,V),...,\lambda_m(z,V)),  
\end{align}
where $e_i$ is the $i$-th elementary symmetric polynomial on $\mathbb{C}^{m}$. Then $f_{\mathcal{J},k}$ is jointly holomorphic in $(z,V)$, and
\[ F_{\mathcal{J}}(z,V,\lambda) = \displaystyle\prod_{j\in \mathcal{J}}^{}(\lambda-\lambda_j(z,V)) = \lambda^{m}+\sum_{k=0}^{m-1}f_{\mathcal{J},k}(z,V)\lambda^k. \]
By our assumption, $ f_{\mathcal{J},k}(z,V^1) \in R$, and Lemma $\ref{supp}$ implies 
\[\mathrm{supp}(f_{\mathcal{J},k}(z,V^1)) \subset \mathcal{S}.\]
Set
\[ c_{k,\alpha}(V^1) = [z^{\alpha}]f_{\mathcal{J},k}(z,V^1), \quad \alpha \in \mathcal{S}.\]
For simplicity of notation, we do not indicate the dependence of \(c_{k,\alpha}\) on \(\mathcal J\). Set $|\mathcal{S}| = s$ and choose $z^1,...,z^s \in \Omega$ such that the matrix 
\[M = ((z^i)^\alpha)_{\substack{1\leq i \leq s
 \\ \alpha \in \mathcal{S}}}\]
is invertible (this holds for generic $(z^1,...,z^s) \in ((\mathbb{C}^*)^d)^s$). Then we have 
\begin{align}\label{value}
 (c_{k,\alpha}(V^1))_{\alpha \in \mathcal{S}} = M^{-1}
\begin{pmatrix}
f_{\mathcal{J},k}(z^1,V^1)
\\\vdots 
\\f_{\mathcal{J},k}(z^s,V^1)
\end{pmatrix}.
\end{align} 
In $\eqref{value}$, we let $V^1$ vary over $U$. This gives a holomorphic function $c_{k,\alpha}(V)$ for every $\alpha \in \mathcal{S}$ and $k = 0,...,m-1$. Define 
\[ \widetilde{f}_{\mathcal{J},k}(z,V) = \displaystyle\sum_{\alpha\in \mathcal{S}}^{}c_{k,\alpha}(V)z^\alpha,\]
then 
\begin{align}\label{coincide}
 f_{\mathcal{J},k}(z,V^1) = \widetilde{f}_{\mathcal{J},k}(z,V^1)   
\end{align}
for $k < m$ and $z \in \Omega$.

For every proper subset $\mathcal{J}$ of $\{1,...,N\}$, define
\[ X_{\mathcal{J}} = \{(z,V) \in \Omega \times U \, | \, f_{\mathcal{J},k}(z,V) = \widetilde{f}_{\mathcal{J},k}(z,V) \,\, \mbox{for every} \,\, 0\leq k < |\mathcal{J}|\}.\]
From $\eqref{coincide}$, for any $V^1 \in U$, there exists $\mathcal{J} \subsetneq \{1,...,N\}$ such that $\Omega \times \{V^1\} \subset X_{\mathcal{J}}$. Hence
\[ \bigcup_{\mathcal{J} \subsetneq \{1,...,N\}, \mathcal{J} \neq \varnothing}X_{\mathcal{J}} = \Omega \times U\]
Since $X_{\mathcal{J}}$ are analytic sets, there exists $\mathcal{J}_0$ such that $X_{\mathcal{J}_0} = \Omega \times U$, which gives the conclusion.
\end{proof}
Next, we prove that the above factor, defined near $V^0$ with coefficients holomorphic in $V$, admits analytic continuation along the loop. Recall that for a loop $\gamma$ in $\mathcal{U}$ based at $V^0$, the functions
\[\lambda_j^{(q)}(z,V) \in \mathcal{O}(\Omega\times B_q), \quad q = 1,...,r\]
form the analytic continuation of $\lambda_j(z,V)$ along $\gamma$, as introduced in Subsection $\ref{analytic}$. We may further require that $B_1 \subset U$.
\begin{lem}\label{contin}
Suppose that for some $\mathcal{J} \subsetneq \{1,...,N\}$,
\[  F(z,V,\lambda) = \displaystyle\prod_{j\in \mathcal{J}}^{}(\lambda-\lambda_j(z,V))\]
is a factor of $D(z,V,\lambda)$ for any $V \in B_1$. Then 
\[ F^{(q)}(z,V,\lambda)= \displaystyle\prod_{j\in \mathcal{J}}^{}(\lambda-\lambda_j^{(q)}(z,V))\]
is a factor of $D(z,V,\lambda)$ for every $V \in B_q$.
\end{lem}
\begin{proof}
We still assume that $\mathcal{J} = \{1,...,m\}$. Set 
\[ F^{(q)}(z,V,\lambda) = \lambda^{m}+\sum_{k=0}^{m-1}f_k^{(q)}(z,V)\lambda^k.\]
As in the proof of Lemma $\ref{holomorphic}$, we can construct $\widetilde{f}_k^{(q)}(z,V) \in \mathcal{O}(\Omega\times B_q)$ such that for any $V^1 \in B_q$, $f_k^{(q)}(z,V^1)$ is a Laurent polynomial supported in $\mathcal{S}$ if and only if 
\begin{align}\label{concide2}
  f_k^{(q)}(z,V^1) = \widetilde{f}_k^{(q)}(z,V^1), \quad \forall \, z \in \Omega,\, k = 0,...,m-1.  
\end{align}
However, $\eqref{concide2}$ holds for every $V^1 \in B_1\cap B_2$. Therefore, by the identity theorem, 
\[f_k^{(2)}(z,V) = \widetilde{f}_{k}^{(2)}(z,V),\quad \forall \, (z,V) \in \Omega \times B_2,\, k = 0,...,m-1.\]
 Using this argument repeatedly, we deduce that $f_k^{(q)}(z,V^1) \in R$ for any fixed $V^1 \in B_q$.

Similarly, on $\Omega \times B_1$, we set 
\[ G(z,V,\lambda) = \displaystyle\prod_{j = m+1}^{N}(\lambda-\lambda_j(z,V)).\]
Then for any $V^1 \in B_1$, $G(z,V^1,\lambda)$ is a factor of $D(z,V^1,\lambda)$ with 
\[D(z,V^1,\lambda) = F(z,V^1,\lambda)G(z,V^1,\lambda).\]
By the above argument, we have $G^{(q)}(z,V^1,\lambda) \in R[\lambda]$ for any $V^1 \in B_q$. Since $\lambda_1^{(q)}(z,V^1),...,\lambda_N^{(q)}(z,V^1)$ are the distinct roots of $D(z,V^1,\lambda)$, we deduce that
\[ D(z,V^1,\lambda) = F^{(q)}(z,V^1,\lambda)G^{(q)}(z,V^1,\lambda), \quad \forall \, V^1 \in B_q.\]
\end{proof}
\begin{thm}\label{main}
If $D(z,V,\lambda)$ is reducible for every $V \in \mathbb{C}^N$, then $D(z,V,\lambda)$ splits over $R$ for any $V \in \mathbb{C}^N$.
\end{thm}
\begin{proof}
We fix $V$ at a point $V^0 \in \mathcal{U}$. From Lemma $\ref{holomorphic}$, there exists a monic factor $F(z,\lambda)$ of $D(z,V^0,\lambda)$ and $\mathcal{J} \subsetneq \{1,...,N\}$ such that 
\[ F(z,\lambda) = \displaystyle\prod_{j\in \mathcal{J}}^{}(\lambda-\lambda_j(z,V^0))\]
for every $z \in \Omega$. For any $\sigma \in S_N$, we choose a loop $\gamma$ in $\mathcal{U}$ with base point $V^0$ such that $\varphi(\gamma) = \sigma$. By Lemma $\ref{contin}$, there is a monic factor $F_\sigma(z,\lambda)$ of $D(z,V^0,\lambda)$ such that
\[ F_\sigma(z,\lambda) = \displaystyle\prod_{j \in \sigma(\mathcal{J})}^{}(\lambda-\lambda_j(z,V^0)) ,\quad \forall\, z\in \Omega.\]

 For any fixed $j_0 \in \{1,...,N\}$, we could choose several elements $\sigma_\beta \in S_N$, $\beta \in \mathcal{B}$, such that 
\begin{align}\label{union}
    \bigcap_{\beta \in \mathcal{B}}\{\sigma_\beta(j)\,|\,j \in \mathcal{J} \} = \{j_0\}.
\end{align} 
From Theorem $\ref{symm}$, we can choose a loop $\gamma_\beta$ with base point $V^0$ such that $\varphi(\gamma_\beta) = \sigma_\beta$. Therefore, for any $\beta \in \mathcal{B}$, there exists a corresponding factor $F_\beta(z,\lambda)$ of $D(z,V^0,\lambda)$ such that
\[ F_{\beta}(z,\lambda) = \displaystyle\prod_{j \in \sigma_{\beta}(\mathcal{J})}^{}(\lambda-\lambda_j(z,V^0)),\quad \forall \, z \in \Omega.\]
Set
\[ Q(z,\lambda) = \gcd_{\beta \in \mathcal{B}}(F_{\beta}),\]
and we may assume that $Q$ is monic (since $F_{\beta}$ is monic). By the B\'{e}zout theorem, there exist $A_\beta \in K[\lambda]$ ($K = \mathbb{C}(z_1,...,z_d)$) such that
\[ \displaystyle\sum_{\beta \in \mathcal{B}}^{}A_\beta(z,\lambda)F_{\beta}(z,\lambda) = Q(z,\lambda).\]
We choose an open set $\Omega_0 \subset \Omega$ such that $A_\beta$ has no poles in $\Omega_0$ for every $\beta \in \mathcal{B}$.  By $\eqref{union}$,
\[Q(z,\lambda_{j_0}(z)) = 0, \quad \forall \, z\in \Omega_0.\] 
The identity theorem implies that this actually holds for every $z \in \Omega$. Since $Q | F_{\beta}$ for every $\beta \in \mathcal{B}$, we have 
\[ \{\lambda \in \mathbb{C}\,|\,Q(z,\lambda) = 0\} \subset \{\lambda_{j}(z)\,| \, j \in \sigma_\beta(\mathcal{J}) \}, \quad \forall \, z \in \Omega,\, \beta \in \mathcal{B}.\]
Therefore, for any fixed $z \in \Omega$, $\lambda_{j_0}(z)$ is the unique root of $Q(z,\lambda)$, and this root is simple, since $D(z,V^0,\lambda)$ has only simple roots. It follows that $\deg(Q) = 1$. Hence, there exists $h_{j_0}(z) \in R$ such that 
\[ Q(z,\lambda) = \lambda-h_{j_0}(z),\]
and $h_{j_0}(z) = \lambda_{j_0}(z,V^0)$ on $\Omega$. Since $j_0$ was arbitrary, we can find $h_j(z) \in R$, $j = 1,..,N$ such that 
\[ \lambda-h_j(z) | D(z,V^0,\lambda), \]
and $h_j(z) = \lambda_j(z,V^0)$ on $\Omega$, thus all $h_j$ are distinct and 
\[ D(z,V^0,\lambda) = \displaystyle\prod_{j = 1}^{N}(\lambda-h_j(z)).\]
Since $V^0$ can be chosen arbitrarily in $\mathcal{U}$, the splitting property holds on a Zariski open set. By Lemma $\ref{closed}$, $D(z,V,\lambda)$ splits over $R$ for any $V \in \mathbb{C}^N$.
\end{proof}

\section{Splitting case}\label{splitcase}
In this section, we assume that $D(z,V,\lambda)$ splits over $R$ for every $V \in \mathbb{C}^N$. We expect that the convergence of eigenvalues in Lemma $\ref{converge}$ can be extended to the convergence of Laurent polynomials that appear in the linear factor of $D(z,V,\lambda)$. From Lemma $\ref{supp}$, the supports of these Laurent polynomials are contained in a common finite set, thus we only need to prove the convergence of certain coefficients, which can be recovered by Cauchy integral formula. The convergence obtained in Lemma $\ref{converge}$ can be applied pointwise under the integral.

\begin{thm}\label{inherit}
   Suppose that 
   \[A(z) = \begin{pmatrix}
 A_{n-1}(z)&\alpha(z) \\ 
\beta^{T}(z) & a_{nn}(z)
\end{pmatrix} \in M_n(R),\]
and set $B(z,V) = \operatorname{diag}(v_1,...,v_n) + A(z)$. Assume that for every $V \in \mathbb{C}^n$, $P(z,V,\lambda) = \det(\lambda I - B(z,V))$ splits over $R$. Then for every $V '  = (v_1,...,v_{n-1})\in \mathbb{C}^{n-1}$, $P_{n-1}(z,V',\lambda) = \det(\lambda I -\mathrm{diag}(v_1,...,v_{n-1}) - A_{n-1}(z))$ splits over $R$.
\end{thm}
\begin{proof}
We fix $v_i=M^i, M \in \mathbb{C}$, for $i=1,\ldots,n-1$, and let $v_n$ vary over $\{z \in \mathbb{C}\, | \, |z| \geq |M|^{n}\}$. Rewrite $B(z,V)$ as $B(z,v_n)$, $P(z,V,\lambda)$ as $P(z,v_n,\lambda)$ and $P_{n-1}(z,V',\lambda)$ as $P_{n-1}(z,\lambda)$. Assume that 
\begin{align}\label{split}
     P(z,v_n,\lambda) = \displaystyle\prod_{j = 1}^{n}(\lambda-\lambda_j(z,v_n)), \quad \lambda_j(z,v_n) \in R.
\end{align}
When $|M|$ is sufficiently large and $z$ is restricted to a compact set, by Gershgorin circle theorem, we can find disks $B_i$ centered at $v_i$, $i = 1,...,n-1$, and a disk $B_n$ centered at $\infty$, such that they are pairwise disjoint, and $B(z,v_n)$ has exactly one eigenvalue in each of the disks. Relabeling if necessary, we may assume for any $z \in \mathbb{T}^d$, 
\[\lambda_j(z,v_n)  \in B_j,\quad \forall \, j \in \{1,...,n\}.\]
By Lemma $\ref{supp}$, there exists a finite set $\mathcal{S} \subset \mathbb{Z}^d$ such that 
\[ \operatorname{supp}(\lambda_j(z,v_n)) \subset \mathcal{S} \]
for every $j = 1,...,n$. Thus, we can write $\lambda_j(z,v_n)$ as 
\[ \lambda_j(z,v_n) = \displaystyle\sum_{\alpha \in \mathcal{S}}c_\alpha^j(v_n)z^\alpha.\]
By the Cauchy integral formula, 
\[c_{\alpha}^j(v_n) =\frac1{(2\pi)^d}\int_{[0,2\pi]^d}\lambda_j(e^{i\theta},v_n)e^{-i\alpha\cdot\theta}\,d\theta.\]
From Lemma $\ref{converge}$, $\lambda_j(z,v_n)$ converges as $v_n \rightarrow \infty$ for any $z \in \mathbb{T}^d$ and $j = 1,...,n-1$. Therefore, by the bounded convergence theorem, $c_\alpha^j(v_n)$ converges as $v_n \rightarrow \infty$. We denote the limit by $c_\alpha^j$, and set
\[ \lambda_j(z) = \displaystyle\sum_{\alpha \in \mathcal{S}}^{}c_\alpha^jz^\alpha.\]
Then for any $z \in (\mathbb{C}^*)^d$,
\begin{align}\label{converge3}
  \lambda_j(z,v_n) \rightarrow \lambda_j(z), \quad \quad \mbox{as}\, \, v_n \rightarrow \infty,  
\end{align}
and this convergence is uniform on any compact set. Also by Lemma $\ref{converge}$, we have 
\begin{align}\label{unb}
 \lambda_n(z,v_n) = v_n + O(1)   
\end{align}
for $z \in \mathbb{T}^d$. Note that 
\[ P(z,v_n,\lambda) = (\lambda-v_n - a_{nn}(z))P_{n-1}(z,\lambda) + R(z,\lambda).\] 
Dividing both sides by $v_n$ and using \eqref{split}, we obtain that
\[(\frac{\lambda}{v_n} - \frac{\lambda_n(z,v_n)}{v_n})\displaystyle\prod_{j = 1}^{n-1}(\lambda- \lambda_{j}(z,v_n)) = (\frac{\lambda}{v_n}-1-\frac{a_{nn}(z)}{v_n})P_{n-1}(z,\lambda) + \frac{R(z,\lambda)}{v_n}.\]
We let $v_n \rightarrow \infty$ and apply $\eqref{converge3}$, $\eqref{unb}$, we obtain
\[ \displaystyle\prod_{j = 1}^{n-1}(\lambda- \lambda_{j}(z)) = P_{n-1}(z,\lambda).\]
Hence $P_{n-1}(z,V',\lambda)$ splits over $R$ for $V' = (M,...,M^{n-1})$. Note that the above argument remains valid for $(v_1,...,v_{n-1})$ near $(M,...,M^{n-1})$. Hence there exists an open set (in the Euclidean topology) in $\mathbb{C}^{n-1}$ such that $P_{n-1}(z,V',\lambda)$ splits over $R$ for every $V'$ in this set. From Lemma $\ref{closed}$, we deduce that $P_{n-1}(z,V',\lambda)$ splits over $R$ for every $V' \in \mathbb{C}^{n-1}$.
\end{proof}
The same argument applies for any $v_i$ in place of $v_n$, and it follows that the characteristic polynomial of the corresponding principal submatrix splits. Applying Theorem $\ref{inherit}$ to $L(z,V)$ repeatedly, we deduce that for any distinct $i,j \in \{1,2,..,N\}$ and the corresponding $2\times2$ principal submatrix
\[ L_{i,j}(z,v_i,v_j)\ = \begin{pmatrix}
 v_i + l_{ii}(z)& l_{ij}(z)\\ 
 l_{ji}(z)&v_j +  l_{jj}(z)
\end{pmatrix},\]
the characteristic polynomial 
\[D_{i,j}(z,v_i,v_j,\lambda) = \det(\lambda I -L_{i,j}(z,v_i,v_j)) \]
splits over $R$ for every $(v_i,v_j) \in \mathbb{C}^2$. By Theorem \ref{simple}, we deduce that for any distinct $i,j \in \{1,...,N\}$ with $C_{ij} \neq \varnothing $, $l_{ij}(z)$ is a monomial, and $l_{ii}(z) = l_{jj}(z)$. Since the quotient graph is connected, for any different $i,j \in \{1,...,N\}$, there exists a path in the quotient graph
\[ u_i =u_{i_1} \rightarrow\cdots\rightarrow u_{i_m} = u_j,\]
such that $C_{i_ki_{k+1}} \neq \varnothing$, $k = 1,...,m-1$. Hence, we have 
\[ l_{i_ki_k}(z) = l_{i_{k+1}i_{k+1}}(z) ,\quad k = 1,...,m-1,\]
which gives $l_{ii}(z) = l_{jj}(z)$. Therefore, we obtain the following theorem:
\begin{thm}\label{fs}
 If $D(z,V,\lambda)$ splits over $R$ for every $V \in \mathbb{C}^N$, then for any distinct $i,j \in \{1,...,N\}$, $l_{ii}(z) = l_{jj}(z)$, and $l_{ij}(z)$ is either identically zero or a monomial.
\end{thm}

Combining Theorem $\ref{main}$ and Theorem $\ref{fs}$, we deduce that if $D(z,V,\lambda)$ is reducible for every potential, then conditions $\mathrm{I}$ and $\mathrm{II}$ hold. To prove Theorem $\ref{main1}$, it remains to show that $\mathrm{quasi}(\mathcal{P}) = 0$ for every cycle $\mathcal{P}$ in the quotient graph.
We first recall the following theorem, the proof can be found in \cite{faust2026generic}.
\begin{thm}[\cite{faust2026generic} Theorem 4.1]\label{irreducible2}
Assume that $\Gamma$ is a $\mathbb{Z}$-periodic graph, and it has a fundamental domain which is minimally connected and there exists exactly one edge between $W$ and $\Gamma\setminus W$.

After fixing non-zero edge weights, $D(x,V,\lambda)$ is irreducible for generic choice of potential.
\end{thm}
\begin{proof}[\bf Proof of Theorem $\ref{main1}$]
Suppose that there exists a cycle $\mathcal{P}$ with 
   \[\operatorname{quasi}(\mathcal{P})\neq 0.\]
Among all such cycles, choose one of minimal length and denote it by $\mathcal {P}_0$. Then $\mathcal{P}_0$ is a simple and chordless cycle, that is, if we set
\[ \mathcal{P}_0 : u_{i_1}  \rightarrow\cdots\rightarrow u_{i_m} \rightarrow u_{i_1},\]
then all the vertices in this cycle are distinct. Moreover, if 
\[j \neq k\pm 1\,(\mathrm{mod \,m}),\]
then $u_{i_j}$ and $u_{i_k}$ are not adjacent in the quotient graph. Indeed, if $\mathcal \mathcal{P}_0$ is not simple or contains a chord, then $\mathcal \mathcal{P}_0$ can be decomposed into two cycles $\mathcal P_1$ and $\mathcal P_2$, each of strictly smaller length, such that 
\[\mathrm{quasi}(\mathcal{P}_1) + \mathrm{quasi}(\mathcal{P}_2)  = \mathrm{quasi}(\mathcal{P}_0).\]
Hence, one of $\mathrm{quasi}(\mathcal{P}_1)$ and $\mathrm{quasi}(\mathcal{P}_2)$ is non-zero. This contradicts the fact that $\mathcal{P}_0$ has a minimal length.

 Now we consider the $\mathbb{Z}^d$-periodic subgraph of $\Gamma$ induced by the vertices in the equivalence classes of $u_{i_1},...,u_{i_m}$, and by a slight abuse of notation, we still denote the corresponding Floquet matrix by $L(z,V)$, where $V = (v_1,...,v_m) \in \mathbb{C}^m$. Next, we delete all edges in this subgraph whose endpoints belong to the same equivalence class. We denote the resulting $\mathbb{Z}^d$-periodic graph by $\Gamma'$ and its Floquet matrix by $L'(z,V)$. From Theorem $\ref{inherit}$, $D(z,V,\lambda)$ splits over $R$ for every $V \in \mathbb{C}^m$. Since condition $\mathrm{I}$ holds, there exists $\phi(z) \in R$ such that 
 \[ L'(z,V) = L(z,V) - \phi(z) I_m.\]
 Therefore, $D'(z,V,\lambda) = \det(\lambda I -L'(z,V))$ also splits over $R$ for any $V \in \mathbb{C}^m$.
 
We now choose a fundamental domain adapted to the cycle $\mathcal{P}_0$. Set
$a=\operatorname{quasi}(\mathcal{P}_0)\neq 0 $, and define
\[
\tau_1=0,\quad
\tau_j=\sum_{k=1}^{j-1}\alpha_{i_k i_{k+1}},
\quad 2\le j\le m.
\]
Write $w_j=u_{i_j}+\tau_j$ and choose $W'=\{w_1,\ldots,w_m\}$ as the new fundamental domain. From the definition of $\tau_j$ and the fact that $\mathcal{P}_0$ is chordless, the edges of $\Gamma'$ are precisely the $\mathbb Z^d$-translates of $(w_j,w_{j+1})$, $1\le j<m$, and $(w_m,w_1+a)$.
Denote the weights of these edges by
\[
c_j=E((w_j,w_{j+1}))\quad (1\le j<m),
\qquad
c_m=E((w_m,w_1+a)).
\]
Changing the fundamental domain only conjugates the Floquet matrix, thus the dispersion polynomial is unchanged. We still use $L'(z,V)$ to denote the Floquet matrix corresponding to \(W'\).

Without loss of generality, we may assume that $a_1\ne0$. Define a weighted $\mathbb Z$-periodic graph $\widetilde\Gamma$ with vertex set
\[
\mathcal V(\widetilde\Gamma)=\{1,\ldots,m\}\times\mathbb Z,
\]
and the edge set $\mathcal{E}(\widetilde{\Gamma})$ consists of
\[
\begin{aligned}
&\bigl((j,n),(j+1,n\bigr),
&&1\le j<m,\quad n\in\mathbb Z,\\
&\bigl((m,n),(1,n+a_1)\bigr),
&&n\in\mathbb Z,
\end{aligned}
\]
with weights $c_j$ and $c_m$, respectively. The $\mathbb Z$-action is given by
\[
r\cdot(j,n)=(j,n+r).
\]
In other words, $\widetilde\Gamma$ is the quotient of $\Gamma'$ by $\{0\}\times\mathbb Z^{d-1}$

The set
\[
\widetilde W=\{(1,0),\ldots,(m,0)\}
\]
is a fundamental domain of $\widetilde{\Gamma}$. Since $a_1\ne0$, the edges in $\widetilde{W}$ are exactly
\[ \bigl((i,0),(i+1,0)\bigr), \quad i = 1,...,m-1.\]
Moreover, the edges 
\[\bigl((m,n),(1,n+a_1)\bigr), \quad n \in \mathbb{Z} \]
form the unique translation orbit among the edges that connect different translates of $\widetilde W$. Thus $\widetilde\Gamma$ satisfies the hypotheses of Theorem 7.3.

For $V=(v_1,\ldots,v_m)$, assign the potential $v_j$ to every vertex $(j,n)$. Let $\widetilde L(x,V)$ be the corresponding Floquet matrix. We observe that
\[
\widetilde L(x,V)=L'((x,1,\ldots,1),V).
\]
Consequently,
\[
\begin{aligned}
\widetilde D(x,V,\lambda)
&:=\det\bigl(\lambda I_m-\widetilde L(x,V)\bigr)\\
&=D'((x,1,\ldots,1),V,\lambda).
\end{aligned}
\]
Since $D'(z,V,\lambda)$ splits over $R$ for every $V\in\mathbb C^m$, specializing $z_1=x$ and $z_2=\cdots=z_d=1$ in its monic linear factors shows that $\widetilde D(x,V,\lambda)$ splits over $\mathbb C[x^{\pm1}]$ for every $V\in\mathbb C^m$. On the other hand, Theorem 7.3 implies that $\widetilde D(x,V,\lambda)$ is irreducible for generic $V\in\mathbb C^m$, which is a contradiction. Therefore, $\operatorname{quasi}(\mathcal{P})=0$ for every cycle $\mathcal{P}$ in the quotient graph.
\end{proof}
\begin{proof}[\bf Proof of Corollary $\ref{Bloch}$]
 Assume that conditions $\mathrm{I}$, $\mathrm{II}$ and $\mathrm{III}$ hold, then 
\[ D(z,V,\lambda) = \displaystyle\prod_{j = 1}^{N}(\lambda-\phi(z) -\lambda_j(V)),\]
where $\phi(z) \in R$ and $\lambda_1(V),...,\lambda_N(V)$ are the eigenvalues of 
\[\mathrm{diag}(v_1,...,v_N) + H.\]
Therefore, the Bloch variety $\mathcal{B}_V$ is irreducible if and only if 
\[ \lambda_1(V) = \lambda_2(V) = \cdots=\lambda_N(V).\]
However, for generic $V \in \mathbb{C}^N$, $\mathrm{diag}(v_1,...,v_N) + H$ has only simple eigenvalues, thus $\mathcal{B}_V$ is reducible for generic $V \in \mathbb{C}^N$. If any of these conditions fails, $D(z,V,\lambda)$ is irreducible for generic $V \in \mathbb{C}^N$, thus $\mathcal{B}_V$ is irreducible for generic $V \in \mathbb{C}^N$.   
\end{proof}
\section*{AI Acknowledgment }ChatGPT 5.6-Sol provided purely algebraic proof for Theorems $\ref{symm}$ and $\ref{main}$, the author reformulates the proofs using methods from analysis, which makes the arguments more natural.
\bibliographystyle{alpha}
\bibliography{reference}
\end{document}